\documentclass[12pt]{amsart}

\usepackage{color}
\usepackage{amstext}
\usepackage{amsfonts}
\usepackage{amsthm}
\usepackage{amsmath}
\usepackage{amssymb}
\usepackage{latexsym}
\usepackage{amsfonts}
\usepackage{graphicx}
\usepackage{texdraw}
\usepackage{graphpap}
\usepackage{subfigure}
\usepackage{enumerate}
\usepackage{here}
\usepackage{tabularx}

\usepackage{mathtools}
\mathtoolsset{showonlyrefs,showmanualtags} 

\usepackage[pagebackref,hypertexnames=false, colorlinks, citecolor=red, linkcolor=red]{hyperref} 
\usepackage[backrefs]{amsrefs}

\makeatletter
\newcommand*{\defeq}{\mathrel{\rlap{%
                     \raisebox{0.3ex}{$\m@th\cdot$}}%
                     \raisebox{-0.3ex}{$\m@th\cdot$}}%
                     =}
										
\newcommand*{\eqdef}{=
										 \mathrel{\rlap{%
                     \raisebox{0.3ex}{$\m@th\cdot$}}%
                     \raisebox{-0.3ex}{$\m@th\cdot$}}%
										}
\makeatother

\usepackage{bbm}   
\newcommand{\unit}{\mathbbm 1} 

\newcommand*\rfrac[2]{{}^{#1}\!/_{#2}}

\newcommand{\mbs}{\mathbb{S}}

\newcommand{\mbr}{\mathbb{R}}

\newcommand{\mcd}{\mathcal{D}}

\newcommand{\mch}{\mathcal{H}}

\newcommand{\ep}{\epsilon}
\newcommand{\lb}{\lambda}
\newcommand{\mfp}{\mathfrak{P}}

\newcommand{\nbmosd}{BMO^2_{\mcd}(\nu)}

\newcommand{\La}{\left\langle }
\newcommand{\Ra}{\right\rangle }

{\end{list}}

\numberwithin{equation}{section}

\newtheorem{thm}{Theorem}[section]
\newtheorem{lm}[thm]{Lemma}

\newtheorem{prop}[thm]{Proposition}
\newtheorem*{prop*}{Proposition}
\newtheorem{defn}[thm]{Definition}

\theoremstyle{remark}

\newtheorem*{rem*}{Remark}

\begin{document}

\title[Two Weight Inequalities for Iterated Commutators with CZOs]{Two Weight Inequalities for Iterated Commutators with Calder{\'o}n-Zygmund Operators}

\author{Irina Holmes}
\address{Irina Holmes, School of Mathematics\\ Georgia Institute of Technology\\ 686 Cherry Street\\ Atlanta, GA USA 30332-0160}
\email{irina.holmes@math.gatech.edu}

\author[Brett D. Wick]{Brett D. Wick$^{\ddagger}$}
\address{Brett D. Wick, Department of Mathematics\\ Washington University - St. Louis\\ One Brookings Drive\\
St. Louis, MO 63130-4899 USA}
\email{wick@math.wustl.edu}
\thanks{$\ddagger$  Research supported in part by a National Science Foundation DMS grants \#0955432 and \#1560955.}

\subjclass[2000]{Primary: 42, 42A, 42B, 42B20, 42B25, 42A50, 42A40, }
\keywords{Commutators, Calder\'on--Zygmund Operators, Bounded Mean Oscillation, Weights}

\begin{abstract}
Given a Calder{\'o}n-Zygmund operator  $T$, a classic result of Coifman-Rochberg-Weiss relates the norm of the commutator $[b, T]$ with the BMO norm of $b$. We focus on a weighted version of this result, obtained  by Bloom and later generalized by Lacey and the authors, which relates $\| [b, T] : L^p(\mathbb{R}^n;\mu) \to L^p(\mathbb{R}^n;\lambda)\|$ to the norm of $b$ in a certain weighted BMO space determined by $A_p$ weights $\mu$ and $\lambda$. We extend this result to higher iterates of the commutator and recover a one-weight result of Chung-Pereyra-Perez in the process. 
\end{abstract}

\maketitle
\setcounter{tocdepth}{1}
\tableofcontents


\section{Introduction and Statement of Main Results}

A Calder\'on--Zygmund operator associated to a kernel $K(x,y)$ is an integral operator:
$$
\mathbf{T}f(x):=\int_{\mathbb{R}^n} K(x,y)f(y)\,dy,\quad x\notin\textnormal{supp} f,
$$
where the kernel satisfies the standard size and smoothness estimates
\begin{gather*}
\left\vert K(x,y)\right\vert  \leq  \frac{C}{\left\vert x-y\right\vert^n}, \\
\left\vert K(x+h,y)-K(x,y)\right\vert +\left\vert K(x,y+h)-K(x,y)\right\vert  \leq  C\frac{\left\vert h\right\vert^{\delta}}{\left\vert x-y\right\vert^{n+\delta}},
\end{gather*}
for all $\left\vert x-y\right\vert>2\left\vert h\right\vert>0$ and a fixed $\delta\in (0,1]$. The prototypes for this important class of operators are the Hilbert transform, in the one-dimensional case, and the Riesz transforms, in the multidimensional case.

Recall that the commutator $[S, T]$ of two operators $S$ and $T$ is defined as
	$[S, T] \defeq ST - TS.$
We are interested in commutators of multiplication by a symbol $b$ with Calder{\'o}n-Zygmund operators $\mathbf{T}$, denoted $[b, \mathbf{T}]$ and defined as:
	$$[b, \mathbf{T}]f \defeq b\mathbf{T}f - \mathbf{T}(bf).$$
In the foundational paper \cite{CRW} Coifman, Rochberg, and Weiss provided a connection between the norm of the commutator $[b,\mathbf{T}]:L^p(\mathbb{R}^n)\to L^p(\mathbb{R}^n)$ and the norm of the function $b$ in $BMO$.  This result was later extended to the case when the commutator acts between two different weighted Lebesgue spaces $L^p(\lambda):=L^p(\mathbb{R}^n;\lambda)$ and $L^p(\mu):=L^p(\mathbb{R}^n;\mu)$. In 1985, Bloom \cite{Bloom} showed that, if $\mu$ and $\lb$ are $A_p$ weights, then $\| [b, H]: L^p(\mu) \to L^p(\lb) \|$ is equivalent to $\|b\|_{BMO(\nu)}$, where $H$ is the Hilbert transform and $BMO(\nu)$ is the weighted BMO space associated with the weight $\nu = \mu^{1/p}\lb^{-1/p}$.  Here, 	$$\|b\|_{BMO(\nu)} \defeq \sup_Q \frac{1}{\nu(Q)} \int_Q |b - \La b\Ra_Q|\,dx < \infty,$$
where $\nu(Q) = \int_Q \,d\nu$, and the supremum is over all cubes $Q$.  When there is no weight involved, we will simply denote this space by $BMO$, which is the classical space of functions with bounded mean oscillation.

A new dyadic proof of Bloom's result was given in \cite{HLW1}. This was then generalized to all Calder{\'o}n-Zygmund operators in \cite{HLW2}, where one of the main results is:

\begin{thm} \label{T:CZOComm1}
Let $\mathbf{T}$ be a Calder{\'o}n-Zygmund operator on $\mbr^n$ and $\mu, \lb \in A_p$ with $1<p<\infty$. Suppose $b \in BMO(\nu)$, where $\nu = \mu^{\frac{1}{p}} \lb^{-\frac{1}{p}}$. Then
	\begin{equation*} 
	\| [b, \mathbf{T}]: L^p(\mu) \to L^p(\lb) \| \leq c \|b\|_{BMO(\nu)},
	\end{equation*}
where $c$ is a constant depending on the dimension $n$, the operator $\mathbf{T}$, and $\mu$, $\lb$, and $p$.
\end{thm}

A natural extension of this is to consider higher iterates of this commutator.  To see how these arise naturally, we follow an argument of Coifman, Rochberg, and Weiss, \cite{CRW}.  For $b\in BMO$ and $r$ sufficiently small, consider the operator:
$$
S_r(f)=e^{rb}\mathbf{T}(e^{-rb}f)
$$
Then it is easy to see that $\left.\frac{d}{dr}S_r(f)\right\vert_{r=0}=[b,\mathbf{T}](f)$ and similarly that we have $\left.\frac{d^n}{dr^n}S_r(f)\right\vert_{r=0}=[b, \ldots, \big[b, [b, \mathbf{T}]\big] \ldots]$ with the function $b$ appearing $n$ times.  For some Calder\'on--Zygmund operator $\mathbf{T}$, let $C_b^1(\mathbf{T}) \defeq [b, \mathbf{T}]$, and
	$$C_b^k(\mathbf{T}) \defeq [b, C_b^{k-1}(\mathbf{T})] \text{, for all integers } k \geq 1.$$
Using weighted theory and the connection between the space $BMO$ and $A_2$ weights, it is then easy to see that the norm of the operator $C_b^k(\mathbf{T})$ on $L^2(\mathbb{R}^n)$ depends on the number of iterates and the norm of the function $b\in BMO$.  At this point, a few natural questions arise:  (1) What is the norm of the $k$th iterate as a function of the norm of $b\in BMO$?  (2) What happens if we attempt to compute the norm of this operator when it acts on $L^p(\mathbb{R}^n;w)$ for a weight $w\in A_p$?  (3) Is there an extension of Theorem \ref{T:CZOComm1} for the iterates?

In the paper \cite{ChungPereyraPerez} Chung, Pereyra, and Perez provide answers to questions (1) and (2) and show that:
$$
\left\| C_b^k(\mathbf{T}):L^2(w)\to L^2(w)\right\|  \leq c \|b\|_{BMO}^k [w]_{A_2}^{k+1},
$$
where $c$ is a constant depending on $n$, $k$ and $\mathbf{T}$.  In fact, they show that, more generally, if $\mathbf{T}$ is any operator bounded on $L^2(\mathbb{R}^n;w)$ with norm $\varphi([w]_{A_2})$ then 
$$
\left\Vert C_b^k(\mathbf{T}):L^2(w)\to L^2(w)\right\Vert\leq \varphi([w]_{A_2})[w]_{A_2}^{k}\left\Vert b\right\Vert_{BMO}^{k}.
$$
However, the two weight extension of Theorem \ref{T:CZOComm1} lies outside the scope of the results in \cite{ChungPereyraPerez}.  Additionally, in \cite[pg. 1166]{ChungPereyraPerez} they ask if it possible to provide a proof of the norm of the iterates of commutators with Calder\'on--Zygmund operators via the methods of dyadic analysis.  The main goal of this paper is to extend Theorem \ref{T:CZOComm1} to the case of iterates, addressing question (3), and in the process show how to answer the question raised in \cite{ChungPereyraPerez}.  This leads to the main result of the paper:

\begin{thm} \label{T:Main}
Let $\mathbf{T}$ be a Calder{\'o}n-Zygmund operator on $\mbr^n$ and $\mu, \lb \in A_p$ with $1<p<\infty$. Suppose $b \in BMO \cap BMO(\nu)$, where $\nu = \mu^{\frac{1}{p}} \lb^{-\frac{1}{p}}$. Then for all integers $k \geq 1$:
	$$\left\| C_b^k(\mathbf{T}) : L^p(\mu) \to L^p(\lb) \right\| \leq c \: \|b\|^{k-1}_{BMO} \|b\|_{BMO(\nu)},$$
where $c$ is a constant depending on $n$, $k$, $\mathbf{T}$, $\mu$, $\lb$, and $p$. 

In particular, if $\mu = \lb = w \in A_2$:
	$$\left\| C_b^k(\mathbf{T}) : L^2(w) \to L^2(w) \right\| \leq c \|b\|_{BMO}^k [w]_{A_2}^{k+1},$$
where $c$ is a constant depending on $n$, $k$ and $\mathbf{T}$.
\end{thm}

The paper is structured as follows. In Section \ref{S:BandN} we discuss the necessary background and notation, such as the Haar system, dyadic shifts, and weighted BMO spaces. Note that most of these concepts were also needed in \cite{HLW2}, and are treated in more detail there. In Section \ref{S:MainProof} we show how, through the Hyt{\"o}nen Representation Theorem, it suffices to prove our main result for dyadic shifts $\mbs^{ij}$. The rest of the paper is dedicated to this. In Section \ref{S:FirstComm} we revisit the two-weight proof for the first commutator $[b, \mathbb{S}^{ij}]$ in \cite{HLW2}, making some definitions which will be useful later, and obtaining the one-weight result. In Section \ref{S:SecondComm} we look at the second iteration $\big[b, [b, \mbs^{ij}]\big]$ -- this will provide the intuition behind the general case of $k$ iterations, and also establish the final tools needed for this. In Section \ref{S:General} we prove the general result.


\section{Background and Notation} \label{S:BandN}


\subsection{The Haar System.} Let  $\mcd^0 = \{ 2^{-k}([0,1)^n + m)\: :\: k\in\mathbb{Z}, m\in \mathbb{Z}^n \}$ be the standard dyadic grid on $\mbr^n$. For any $\omega = (\omega_j)_{j\in\mathbb{Z}} \in (\{0, 1\}^n)^{\mathbb{Z}}$, we let $\mcd^{\omega} \defeq \{Q \stackrel{\cdot}{+} \omega \: : \: Q\in\mcd^0\}$ be the translate of $\mcd^0$ by $\omega$, where
	$$Q \stackrel{\cdot}{+} \omega \defeq Q + \sum_{j: 2^{-j}< l(Q)} 2^{-j}\omega_j,$$
where $l(Q)$ denotes the side length of any cube $Q$ in $\mathbb{R}^n$. Every dyadic grid $\mcd^{\omega}$ is characterized by two fundamental properties, namely (1) For every $P, Q \in \mcd^{\omega}$, $P \cap Q \in \{ P, Q, \emptyset \}$, and (2)  For every fixed $k \in\mathbb{Z}$, the cubes $Q\in\mcd^{\omega}$ with $l(Q) = 2^{-k}$ partition $\mbr^n$.
Let $\mcd$ be a fixed dyadic grid, $Q\in\mcd$, and $k$ be a non-negative integer. We let $Q^{(k)}$ denote  the $k^{\text{th}}$ ancestor of $Q$ in $\mcd$, i.e. the unique element of $\mcd$ with side length $2^k l(Q)$ that contains $Q$, and $Q_{(k)}$ denote the collection of $k^{\text{th}}$ descendants of $Q$ in $\mcd$, i.e. the $2^{kn}$ disjoint subcubes of $Q$ in $\mcd$ with side length $2^{-k} l(Q)$. 	

The Haar system on $\mcd$ is defined by associating $2^n$ Haar functions to every $Q = I_1 \times \cdots \times I_n \in \mcd$, where each $I_i$ is a dyadic interval in $\mbr$ with length $l(Q)$:
	$$ h_Q^{\ep} (x) \defeq \prod_{i=1}^n h_{I_i}^{\ep_i}(x_i),$$
for all $x = (x_1, \ldots, x_n) \in \mbr^n$, where $\ep \in \{0, 1\}^n$ is called the signature of $h_Q^{\ep}$, and $h_{I_i}^{\ep_i}$ is one of the one-dimensional Haar functions:
	$$ h_I^0 \defeq \frac{1}{\sqrt{|I|}} (\unit_{I_{-}} - \unit_{I_{+}}); \:\:\: h_I^1 \defeq \frac{1}{\sqrt{|I|}} \unit_I.$$
We write $\ep = 1$ when $\ep_i = 1$ for all $i$. In this case, $h_Q^1 = |Q|^{-\frac{1}{2}} \unit_Q$ is said to be non-cancellative, while all the other $2^n-1$ Haar functions associated with $Q$ are cancellative. Moreover, the cancellative Haar functions on a fixed dyadic grid form an orthonormal basis for $L^2(\mbr^n)$. We then write for any $f \in L^2(\mbr^n)$:
	$$ f = \sum_{Q\in\mcd, \ep\neq 1} \widehat{f}(Q,\ep) h_Q^{\ep}, $$
where $\widehat{f}(Q, \ep) \defeq \La f, h_Q^{\ep}\Ra$ and $\La \cdot,\cdot\Ra$ denotes the usual inner product in $L^2(\mbr^n)$.	Then
	$$ \La f\Ra_Q = \sum_{R\in\mcd, R\supsetneq Q; \ep\neq 1} \widehat{f}(R,\ep) h_R^{\ep}(Q), $$
where $\La f\Ra_Q \defeq |Q|^{-1} \int_Q f\,dx$ denotes the average of $f$ over $Q$.
	

\subsection{$A_p$ weights.} By a weight on $\mbr^n$ we mean an almost everywhere positive, locally integrable function $w$. For some $1<p<\infty$ with H{\"o}lder conjugate $q$, we say that a weight $w$ belongs to the Muckenhoupt $A_p$ class if
	$$[w]_{A_p} \defeq \sup_{Q} \La w\Ra_Q \La w^{1-q}\Ra_Q^{p-1} < \infty,$$
where  the supremum is over all cubes in $\mbr^n$. We let $w' \defeq w^{q-1}$, the `conjugate' weight to $w$. Then $w\in A_p$ if and only if $w' \in A_q$, with $[w']_{A_q} = [w]_{A_p}^{q-1}$. 

For a weight $w$ and $1<p<\infty$, let $L^p(w)$ denote the usual $L^p$ space with respect to the measure $dw = w\,dx$, i.e. the space of all functions $f$ such that $\|f\|_{L^p(w)}^p = \int_{\mbr^n} |f|^p\,dw < \infty$. If $w\in A_p$, we then have the duality $(L^p(w))^* \equiv L^q(w')$, in the sense that
	$$ \|f\|_{L^p(w)} = \sup\{ |\La f, g\Ra| \: :\: g \in L^q(w'), \|g\|_{L^q(w')} \leq 1\}.$$

We review some of the crucial properties of $A_p$ weights, starting the maximal function: 
	$$Mf \defeq \sup_{Q} \left(\La |f|\Ra_Q \unit_Q \right),$$
where again the supremum is over all cubes $Q$ in $\mbr^n$. If $w\in A_p$, then the following bound is sharp \cites{Muck, Buckley} in the exponent of $[w]_{A_p}$:
\begin{equation} \label{E:MaxOneWeight}
	\|M\|_{L^p(w)} \lesssim [w]_{A_p}^{\frac{1}{p-1}},
\end{equation}
where for some quantities $A$ and $B$, ``$A \lesssim B$'' denotes $A \leq CB$ for some absolute constant $C$. Another important tool is the dyadic square function:
	$$(S_{\mcd}f)^2 \defeq \sum_{Q\in\mcd,\ep\neq 1} |\widehat{f}(Q,\ep)|^2 \frac{\unit_Q}{|Q|},$$
for which we have the sharp \cite{CruzClassicOps} one-weight inequality:	
\begin{equation} \label{E:SFOneWeight}
	\|S_{\mcd}\|_{L^p(w)} \lesssim [w]_{A_p}^{\max\left(\frac{1}{2}, \frac{1}{p-1}\right)}.
\end{equation}
For a dyadic grid $\mcd$ on $\mbr^n$ and a pair $(i, j)$ of non-negative integers, define a shifted dyadic square function:
	\begin{equation} \label{E:ShiftedDSFDef}
	\big(\widetilde{S_{\mcd}}^{i,j} f\big)^2 \defeq \sum_{Q\in\mcd, \ep\neq 1} \bigg(\sum_{P \in (Q^{(j)})_{(i)}} |\widehat{f}(P,\ep)| \bigg)^2 \frac{\unit_Q}{|Q|}.
	\end{equation}
The following was proved in \cite[Lemma 2.2]{HLW2}:
\begin{equation} \label{E:ShiftedSF}
\|\widetilde{S_{\mcd}}^{i,j}\|_{L^2(w)} \lesssim 2^{\frac{n}{2}(i+j)} [w]_{A_2}.
\end{equation}

Lastly, we recall the extrapolation property of $A_p$ weights \cite{Extrapolation}.  Suppose an operator $T$ satisfies:	
	$$\|Tf\|_{L^2(w)} \lesssim A [w]_{A_2}^{\alpha} \|f\|_{L^2(w)}$$
for all $w\in A_2$, for some fixed $A> 0$ and $\alpha>0$. Then:
	$$\|Tf\|_{L^p(w)} \lesssim A [w]_{A_p}^{\alpha\max\left(1, \frac{1}{p-1}\right)}\|f\|_{A_p},$$
for all $1<p<\infty$ and all $w\in A_p$.


\subsection{Weighted BMO} Let $w$ be a weight on $\mbr^n$. The weighted BMO space $BMO(w)$ is the space of all locally integrable functions $b$ such that
	$$\|b\|_{BMO(w)} \defeq \sup_Q \frac{1}{w(Q)} \int_Q |b - \La b\Ra_Q|\,dx < \infty,$$
where $w(Q) = \int_Q \,dw$, and the supremum is over all cubes $Q$. Note that if we take $w=1$ we obtain the usual space of functions with bounded mean oscillation, which we simply denote by $BMO$. If  $w\in A_p$, it was shown in \cite{MuckWheeden} that $\|\cdot\|_{BMO(w)}$ is equivalent to the norm $\|\cdot\|_{BMO^q(w)}$, defined as
	$$\|b\|^q_{BMO^q(w)} \defeq \sup_Q \frac{1}{w(Q)} \int_Q |b - \La b\Ra_Q|^q\,dw'.$$
Given a dyadic grid $\mcd$, we define the dyadic versions of these spaces, $BMO_{\mcd}(w)$ and $BMO^q_{\mcd}(w)$, by taking the supremum over $Q\in\mcd$ instead. 

Now suppose $\mu,\lb\in A_p$ and define $\nu \defeq \mu^{\frac{1}{p}} \lb^{-\frac{1}{p}}$. As shown in \cite{HLW2}, $\nu$ is then an $A_2$ weight. The following inequality will be very useful:
	\begin{equation}\label{E:NuH1BMODuality}
	|\La b,\Phi\Ra| \lesssim [\nu]_{A_2} \|b\|_{BMO^2_{\mcd}(\nu)} \|S_{\mcd}\Phi\|_{L^1(\nu)}.
	\end{equation}
This in fact holds for all $A_2$ weights $w$, and comes from a duality relationship between $BMO^2_{\mcd}(w)$ and the dyadic weighted Hardy space $\mch^1_{\mcd}(w)$. See \cite[Section 2.6]{HLW2} for details.

\subsection{Paraproducts.} Recall the paraproducts with symbol $b$ on $\mbr$:
	$\Pi_bf = \sum_{I\in\mcd} \widehat{b}(I)\La f\Ra_I h_I$  and  $\Pi^*_bf = \widehat{b}(I)\widehat{f}(I) |I|^{-1}\unit_I$. 
These are most useful in dyadic proofs due to the identity: $bf = \Pi_b f + \Pi^*_b f + \Pi_f b$. To generalize this property to $\mbr^n$, we define the multidimensional paraproducts below.

\begin{defn} \label{D:Paraprods}
For a fixed dyadic grid $\mcd$ on $\mbr^n$, define the following paraproduct operators with symbol $b$:
$$\Pi_b f \defeq \sum_{Q \in \mcd, \ep \neq 1} \widehat{b}(Q, \ep) \left<f\right>_Q h_Q^{\ep}, \:\: \:\:\: 
	\Pi_b^* f \defeq \sum_{Q \in \mcd, \ep \neq 1} \widehat{b}(Q, \ep) \widehat{f}(Q, \ep) \frac{\unit_Q}{|Q|},$$
and
$$\Gamma_b f \defeq \sum_{Q \in \mcd} \sum_{\stackrel{\ep, \eta \neq 1}{\ep \neq \eta}} \widehat{b}(Q, \ep) \widehat{f}(Q, \eta) \frac{1}{\sqrt{|Q|}} h_Q^{\ep + \eta},$$
where for every $\ep,\eta \in \{0,1\}^n$, $\ep+\eta$ is defined by letting $(\ep+\eta)_i$ be $0$ if $\ep_i\neq\eta_i$ and $1$ otherwise.
\end{defn}

Then  $bf = \Pi_b f + \Pi_b^*f + \Gamma_b f + \Pi_f b$. Note that, while the first two paraproducts above reduce to the standard one-dimensional ones when $n=1$, the third paraproduct $\Gamma_b$ vanishes in this case. This third paraproduct comes from the fact that $h_Q^{\ep}h_Q^{\eta} = |Q|^{-\rfrac{1}{2}} h_Q^{\ep+\eta}$.

For ease of notation later, we denote:
	$$\|T\|_{L^p(w;\:v)} \defeq \left\|T:L^p(w) \to L^p(v)\right\|,$$
the operator norm between two weighted $L^p$-spaces.  And, when $w=v$ we will frequently write 
	$$\|T\|_{L^p(w)} \defeq \left\|T:L^p(w) \to L^p(w)\right\|.$$

The following two-weight result was proved in \cite[Theorem 3.1]{HLW2}. We recall first that the adjoints of $\Pi_b$, $\Pi^*_b$, and $\Gamma_b$ as $L^p(\mu) \to L^p(\lb)$ operators are $\Pi^*_b$, $\Pi_b$, and $\Gamma_b$ as $L^q(\lb') \to L^q(\mu')$ operators, respectively.
\begin{thm} \label{T:Paraprod2WtBds}
Let $\mu, \lb \in A_p$ for some $1<p<\infty$, $\nu = \mu^{\frac{1}{p}}\lb^{-\frac{1}{p}}$, and suppose $b \in \nbmosd$ for a fixed dyadic grid $\mcd$ on $\mbr^n$. Then:
	\begin{align} 
\label{E:PibUBd}
	\left\|\Pi_b\right\|_{L^p(\mu;\:\lb)}  = \left\|\Pi^*_b\right\|_{L^q(\lb';\:\mu')} &\lesssim c \|b\|_{BMO^2_{\mcd}(\nu)},\\
 \label{E:PibStarUBd}
	\left\|\Pi^*_b\right\|_{L^p(\mu;\:\lb)}  = \left\|\Pi_b\right\|_{L^q(\lb';\:\mu')} & \lesssim c \|b\|_{BMO^2_{\mcd}(\nu)},\\
\label{E:GammabUBd}
	\left\|\Gamma_b\right\|_{L^p(\mu;\:\lb)} = \left\|\Gamma_b\right\|_{L^q(\lb';\:\mu')}  & \lesssim  c \|b\|_{BMO^2_{\mcd}(\nu)},
	\end{align}
where, in each case, $c$ denotes a constant depending on $\mu$, $\lb$, and $p$.
\end{thm}


\subsection{Dyadic Shifts.} Let $i,j$ be non-negative integers and $\mcd$ a dyadic grid on $\mbr^n$. A dyadic shift operator with parameters $(i,j)$ is an operator of the form:
	\begin{equation} \label{E:SijDef}
	\mbs_{\mcd}^{ij}f \defeq \sum_{R \in \mcd} \sum_{P \in R_{(i)} , Q \in R_{(j)}} \sum_{\ep,\eta \in \{0,1\}^n} a^{\ep\eta}_{PQR} \widehat{f}(P,\ep) h_Q^{\eta},
	\end{equation}
where $a^{\ep\eta}_{PQR}$ are coefficients with $a^{\ep\eta}_{PQR} \leq |R|^{-1}\sqrt{|P||Q|}$. The shift is said to be cancellative if all Haar functions in its definition are cancellative, that is $a^{\ep\eta}_{PQR} = 0$ whenever $\ep = 1$ or $\eta = 1$. Otherwise, it is called non-cancellative. 

The following weighted inequality for dyadic shifts, which can be found in \cites{HytLacey,Lacey,TreilSharpA2}, will be extremely useful:

\begin{thm} \label{T:SijWeighted}
Let $\mbs_{\mcd}^{ij}$ be a dyadic shift operator. Then for any weight $w\in A_2$:
	\begin{equation} \label{E:SijWeighted}
	\left\|\mbs_{\mcd}^{ij}\right\|_{L^2(w)}  \lesssim \kappa_{ij} [w]_2, 
	\end{equation}
where $\kappa_{ij} \defeq \max\{i, j, 1\}$ is the complexity of the shift.
\end{thm}

As a first application of this result, we observe that that all paraproducts in Definition \ref{D:Paraprods} can be expressed in terms of dyadic shifts with parameters $(0, 0)$, that is, shifts of the form
	$$\mbs^{00}f = \sum_{Q\in\mcd; \ep,\eta\in\{0,1\}^n} a_Q^{\ep\eta} \widehat{f}(Q,\ep) h_Q^{\eta};\:\: |a_Q^{\ep\eta}| \leq 1.$$
For instance, we may write $\Pi_b = \|b\|_{BMO^2_{\mcd}} \mbs^{00}$, where $a^{\ep\eta}_Q = 0$ if $\ep\neq 1$ or $\eta=1$, and $a^{\ep\eta}_Q = \widehat{b}(Q,\eta) |Q|^{-1/2} \|b\|_{BMO^2_{\mcd}}^{-1}$ if $\ep=1$ and $\eta\neq 1$. Similar expressions can be obtained for the other two paraproducts.  Then, if $P_b$ is any one of  $\Pi_b$, $\Pi^*_b$, or $\Gamma_b$, it follows from \eqref{E:SijWeighted} that	
	\begin{equation} \label{E:Paraprod1WtBds}
	\|P_b\|_{L^2(w)} \lesssim \|b\|_{BMO^2_{\mcd}} [w]_{A_2},
	\end{equation}
for any $w \in A_2$. These one-weight inequalities for paraproducts were obtained in \cite{Beznosova} for the one-dimensional case $n=1$, and, using the Wilson Haar basis, in \cite{Chung} for $n \geq 1$.

We can also use dyadic shifts to recover the one-weight bound for the martingale transform:
	$$T_{\sigma}f \defeq \sum_{Q\in\mcd,\ep\neq 1} \sigma_{Q,\ep} \widehat{f}(Q,\ep) h_Q^{\ep},$$
where $|\sigma_{Q,\ep}| \leq 1$ for all $Q\in\mcd$ and $\ep\neq 1$. For $w\in A_2$, all martingale transforms $T_{\sigma}$ are uniformly bounded on $L^2(w)$. In particular, there is a universal constant $C$ such that
\begin{equation} \label{E:MartTBd}
\|T_{\sigma}\|_{L^2(w)} \leq C [w]_{A_2}, 
\end{equation}
for all $\sigma$. This result, obtained in \cite{Wittwer} for the one-dimensional case, trivially follows from the observation that $T_{\sigma} = \mbs^{00}$, where $a^{\ep\eta}_{Q}$ is defined to be $\sigma_{Q,\ep}$ if $\ep=\eta\neq 1$ and $0$ otherwise.

The following simple consequence of this fact will come in handy later:

\begin{prop} \label{P:BMOSigma}
Let $w \in A_2$, $b \in BMO^2_{\mcd}(w)$, and $T_{\sigma}$ be a martingale transform. Then $T_{\sigma}b \in BMO^2_{\mcd}(w)$, with
	\begin{equation} \label{E:BMOSigma}
	\|T_{\sigma}b\|_{BMO^2_{\mcd}(w)} \lesssim [w]_{A_2} \|b\|_{BMO^2_{\mcd}(w)}.
	\end{equation}
\end{prop}

\begin{proof}
It is easy to observe that
	$\unit_Q(T_{\sigma}b - \La T_{\sigma}b\Ra_{Q}) = T_{\sigma}\big( \unit_Q (b - \La b\Ra_Q) \big)$,
and so
	$$ \|T_{\sigma}b\|_{BMO^2_{\mcd}(w)} = \sup_{Q\in\mcd} \frac{1}{w(Q)^{\rfrac{1}{2}}} \left\| T_{\sigma} \big( \unit_Q (b - \La b\Ra_Q) \big) \right\|_{L^2(w^{-1})} \lesssim [w]_{A_2} \|b\|_{BMO^2_{\mcd}(w)}.$$
\end{proof}


\section{Proof of The Main Result} \label{S:MainProof}

As in the proof of Theorem \ref{T:CZOComm1} in \cite{HLW2}, the backbone of our proof of Theorem \ref{T:Main} is the celebrated Hyt{\"o}nen Representation Theorem \cites{HytRepOrig, HytRep, HytPerezTV}, which we state below.

\begin{thm} \label{T:HRP}
Let $\mathbf{T}$ be a Calder{\'o}n-Zygmund operator associated with a $\delta$-standard kernel. Then there exist dyadic shift operators $\mbs_{\omega}^{ij}$ with parameters $(i, j)$ for all non-negative integers $i, j$ such that
	$$\La \mathbf{T}f, g\Ra = c\: \mathbb{E}_{\omega} \sum_{i,j=0}^{\infty} 2^{- \kappa _{i,j}\frac{\delta}{2}} \La \mbs_{\omega}^{ij}f, g\Ra,$$
for all bounded, compactly supported functions $f$ and $g$, where $c$ is a constant depending on the dimension $n$ and on $\mathbf{T}$. Here all $\mbs^{ij}_{\omega}$ with $(i,j) \neq (0,0)$ are cancellative, but the shifts $\mbs_{\omega}^{00}$ may not be cancellative.
\end{thm}

It is easy to see that
	\begin{equation} \label{E:Main_t1}
	\La C_b^k(\mathbf{T})f, g\Ra = c \: \mathbb{E}_{\omega} \sum_{i,j=0}^{\infty} 2^{- \kappa _{i,j}\frac{\delta}{2}}  \La C_b^k(\mbs^{ij}_{\omega})f, g\Ra,
	\end{equation}
for all integers $k\geq 1$. Thus it suffices to show that $C_b^k(\mbs^{ij}_{\omega})$ are uniformly bounded, regardless of $\omega$, with bounds that depend at most polynomially on $\kappa_{ij}$. Since our arguments will be independent of choice of $\omega$, we fix a dyadic grid $\mcd$ and suppress the $\omega$ subscript in what follows. We claim that:

\begin{thm} \label{T:BigThm_Sij}
Let $\mu, \lb \in A_p$ for some $1<p<\infty$ and $b \in BMO^2_{\mcd}(\nu) \cap BMO^2_{\mcd}$, where $\nu = \mu^{\frac{1}{p}} \lb^{-\frac{1}{p}}$. For any pair $(i, j)$ of non-negative integers, let $\mbs^{ij} \defeq \mbs^{ij}_{\mcd}$ be a dyadic shift as in the Hyt{\"o}nen Representation Theorem. Then for any integer $k \geq 1$:
	\begin{equation}
	\left\| C_b^k(\mbs^{ij}) \right\|_{L^p(\mu;\:\lb)} \leq c\: \kappa^k_{ij} \|b\|^{k-1}_{BMO^2_{\mcd}} \|b\|_{BMO^2_{\mcd}(\nu)},
	\end{equation}
where $c$ is a constant depending on $n$, $p$, $\mu$, $\lb$, and $k$. In particular, if $\mu = \lb = w \in A_2$:
	\begin{equation}
	\left\| C_b^k(\mbs^{ij}) \right\|_{L^2(w)} \leq c\: \kappa^k_{ij} \|b\|^k_{BMO^2_{\mcd}} [w]_{A_2}^{k+1},
	\end{equation}
where $c$ is a constant depending on $n$ and $k$.
\end{thm}

Then for all $\omega$:
	$$\left\| C_b^k(\mbs^{ij}_{\omega}) \right\|_{L^p(\mu;\:\lb)} \leq c\: \kappa_{ij}^k \|b\|^{k-1}_{BMO} \|b\|_{BMO(\nu)} \text{, and }
	  \left\|C_b^k(\mbs^{ij}_{\omega}) \right\|_{L^2(w)} \leq c\: \kappa^k_{ij} \|b\|_{BMO}^k [w]^{k+1}_{A_2},$$
where we used the equivalence of $BMO(\nu)$ and $BMO^2(\nu)$ norms. Then \eqref{E:Main_t1} gives that
	$$\left\|C_b^k(\mathbf{T})\right\|_{L^p(\mu;\:\lb)} \leq c\: \|b\|^{k-1}_{BMO} \|b\|_{BMO(\nu)} \sum_{i,j=0}^{\infty} 2^{- \kappa _{i,j}\frac{\delta}{2}} \kappa_{ij}^k
		\leq c' \:\|b\|^{k-1}_{BMO} \|b\|_{BMO(\nu)}.$$
The one-weight result in Theorem \ref{T:Main} follows similarly. The rest of the paper is dedicated to proving Theorem \ref{T:BigThm_Sij}.


\section{The Commutator $[b, \mbs^{ij}]$ Revisited} \label{S:FirstComm}

Recall that the product of two functions can be formally decomposed in terms of the paraproducts in Definition \ref{D:Paraprods} as
	$bf = \mathfrak{P}_bf + \Pi_f b,$
where 
	$$\mathfrak{P}_b \defeq \Pi_b + \Pi^*_b + \Gamma_b.$$
Consequently, the commutator $[b, T]$ with an operator $T$ can be expressed as:
	$$C_b^1(T)f = [b, T]f = [\mathfrak{P}_b, T]f + \left( \Pi_{Tf}b - T\Pi_f b \right).$$
Proving inequalities for $[b, T]$ via dyadic methods usually involves proving some appropriate bounds for the paraproducts -- from which the boundedness of the first term $[\mathfrak{P}_b, T]$ usually follows -- and then treating the `remainder term' $\mathcal{R}_1f = \Pi_{Tf}b - T\Pi_f b$ separately. 

Now, remark that if we consider the second iteration:
	$$C_b^2(T)f = [b, C_b^1(T)]f = [\mathfrak{P}_b, C_b^1(T)]f + \left( \Pi_{C_b^1(T) f}b - C_b^1(T)\Pi_f b \right),$$
we will encounter the term $\Pi_{\mathcal{R}_1f}b - \mathcal{R}_1\Pi_f b$. We can already see that more compact notation for repeatedly performing the operation $T \mapsto \Pi_{T\cdot}b - T\Pi_{\cdot}b$ would be useful. 

\begin{defn} \label{D:Theta_b}
Given an operator $T$ on some function space and a function $b$, define the operator $\Theta_b(T)$ by:
	$$\Theta_b(T)f \defeq \Pi_{Tf} b - T \Pi_f b.$$
More generally, let $\Theta_b^0(T) \defeq T$, and
	$$\Theta_b^k(T) \defeq \Theta_b\big( \Theta_b^{k-1}(T) \big),$$
for all integers $k \geq 1$.
\end{defn}

Using this notation, for an operator $T$,
	\begin{equation} \label{E:CommTh_bDec}
	[b, T] = [\mathfrak{P}_b, T] + \Theta_b(T),
	\end{equation}
In particular,
	\begin{equation} \label{E:Comm1_Dec}
	C_b^1(\mbs^{ij}) = [\mfp_b, \mbs^{ij}] + \Theta_b(\mbs^{ij}),
	\end{equation}
so
	\begin{align} 
	\label{E:Cb1}
	\| C_b^1(\mbs^{ij}) \|_{L^p(\mu;\:\lb)} & \leq \|\mfp_b\|_{L^p(\mu;\:\lb)} \big( \|\mbs^{ij}\|_{L^p(\mu)} + \|\mbs^{ij}\|_{L^p(\lb)} \big) + \|\Theta_b(\mbs^{ij})\|_{L^p(\mu;\:\lb)}\\
		& \lesssim \kappa_{ij} C(\mu, \lb, p) \|b\|_{BMO^2_{\mcd}(\nu)} + \|\Theta_b(\mbs^{ij})\|_{L^p(\mu;\:\lb)},
	\end{align}
where the first term is bounded using Theorem \ref{T:Paraprod2WtBds} and \eqref{E:SijWeighted}. Letting $\mu = \lb = w \in A_2$ in \eqref{E:Cb1} and using the one-weight bounds \eqref{E:Paraprod1WtBds} for the paraproducts:
	\begin{align}
	\|C_b^1(\mbs^{ij})\|_{L^2(w)} & \leq 2 \|\mfp_b\|_{L^2(w)} \|\mbs^{ij}\|_{L^2(w)} + \|\Theta_b(\mbs^{ij})\|_{L^2(w)}\\
		& \lesssim \kappa_{ij}\|b\|_{BMO^2_{\mcd}} [w]_{A_2}^2 + \|\Theta_b(\mbs^{ij})\|_{L^2(w)}.
	\end{align}
So it remains to bound the remainder term.   And, based on analysis that will come later, we in fact need to control certain iterates of the remainder term, leading to the following claim:


\begin{prop} \label{P:Comm1}
Under the same assumptions as Theorem \ref{T:BigThm_Sij}, for all integers $k \geq 1$:
\begin{equation} \label{E:Comm1_2Wt}
\left\| \Theta_b^k (\mbs^{ij}) \right\|_{L^p(\mu;\:\lb)} \leq c \kappa^k_{ij}  \|b\|^{k-1}_{BMO^2_{\mcd}} \|b\|_{BMO^2_{\mcd}(\nu)},
\end{equation}
where $c$ is a constant depending on $n$, $p$, $\mu$, $\lb$, and $k$, and
\begin{equation} \label{E:Comm1_1Wt}
\left\| \Theta_b^k (\mbs^{ij}) \right\|_{L^2(w)} \leq c \kappa^k_{ij} \| b \|^k_{BMO^2_{\mcd}} [w]^{k+1}_{A_2},
\end{equation}
where $c$ is a constant depending on $n$ and $k$.
\end{prop}

Obviously, letting $k=1$ in this proposition yields the results in Theorem \ref{T:BigThm_Sij} with $k = 1$.
The first result \eqref{E:Comm1_2Wt} was proved for $k=1$ in \cite{HLW2}. In this section we revisit this proof in order to obtain the one-weight result. The latter will follow directly from the two-weight proof in the case of cancellative shifts $\mbs^{ij}$ with $(i, j) \neq (0, 0)$, but the case $(i, j) = (0, 0)$ will require some care. However, as we shall see in the next section, the tools we introduce here lay most of the groundwork for the iterated commutators.


\subsection{The Cancellative Shifts.} We showed in \cite{HLW2} that
\begin{equation} \label{E:ThbSij}
\Theta_b (\mbs^{ij}) = \sum_{\substack{R\in\mcd \\ \ep,\eta\neq 1}} \sum_{\substack{P \in R_{(i)} \\ Q\in R_{(j)}}} a^{\ep\eta}_{PQR} \widehat{f}(P,\ep) \big( \La b\Ra_Q - \La b\Ra_P \big) h_Q^{\eta},
\end{equation}
whenever $(i, j) \neq (0, 0)$ and the dyadic shift is cancellative. Then, assuming $i \leq j$, we expressed $\Theta_b(\mbs^{ij})$ as:
$$\Theta_b(\mbs^{ij}) = \sum_{l=1}^j A_l - \sum_{m=1}^i B_m,$$
for some certain operators $A_l$, $B_m$. Using the weighted $\mch^1 - BMO$ duality statement in \eqref{E:NuH1BMODuality}, we showed that these operators satisfy:
\begin{align}
\label{E:Al}
\|A_l\|_{L^p(\mu;\: \lb)} & \leq [\nu]_{A_2} \|b\|_{BMO^2_{\mcd}(\nu)} 2^{-\frac{n}{2}(i+j-l)} \|M\|_{L^q(\lb')} \| \widetilde{S_{\mcd}}^{i, j-l} \|_{L^p(\mu)},  \\
\label{E:Bm}
\|B_m\|_{L^p(\mu;\: \lb)} & \leq [\nu]_{A_2} \|b\|_{BMO^2_{\mcd}(\nu)} 2^{-\frac{n}{2}(i+j-m)} \|M\|_{L^p(\mu)} \| \widetilde{S_{\mcd}}^{j, i-m} \|_{L^q(\lb')},  
\end{align}
for all integers $1 \leq l \leq j$ and $1 \leq m \leq j$, where $\widetilde{S_{\mcd}}^{i,j}$ is the shifted dyadic square function in \eqref{E:ShiftedDSFDef}. From here, \eqref{E:Comm1_2Wt} follows easily from \eqref{E:MaxOneWeight} and \eqref{E:ShiftedSF}.

Now, if we let $\mu = \lb = w \in A_2$,  \eqref{E:Al} becomes:
	\begin{align*}
	\|A_l\|_{L^2(w)} &\leq \|b\|_{BMO^2_{\mcd}} 2^{-\frac{n}{2}(i+j-l)} \|M\|_{L^2(w^{-1})} \|\widetilde{S_{\mcd}}^{i, j-l}\|_{L^2(w)} \\
	& \lesssim \|b\|_{BMO^2_{\mcd}} 2^{-\frac{n}{2}(i+j-l)} [w^{-1}]_{A_2} 2^{\frac{n}{2}(i+j-l)} [w]_{A_2}\\
	& = \|b\|_{BMO^2(\mcd)} [w]^2_{A_2}.
	\end{align*}
Similarly, we obtain $\|B_m\|_{L^2(w)} \lesssim \|b\|_{BMO^2_{\mcd}} [w]^2_{A_2}$ from \eqref{E:Bm}, and \eqref{E:Comm1_1Wt} follows. The proof for $i \geq j$ is symmetrical.


\subsection{The case $i = j = 0$.} As shown in \cite{HytRepOrig}, the non-cancellative shift $\mbs^{00}$ is of the form
	\begin{equation} \label{E:00HRT}
	\mbs^{00} = \mbs_c^{00} + \Pi_a + \Pi^*_d,
	\end{equation}
where $\mbs_c^{00}$ is a \textit{cancellative} shift with parameters $(0, 0)$, and $\Pi_a$, $\Pi^*_d$ are paraproducts with symbols $a, d \in BMO_{\mcd}$ and $\|a\|_{BMO_{\mcd}} \leq 1$, $\|d\|_{BMO_{\mcd}} \leq 1$.


In \cite[Section 5.2]{HLW2} we show that $\Theta_b(\mbs_c^{00}) = 0$ and
	\begin{align} 
	\Theta_b(\Pi_a) &= \Pi_a \Pi_b + \Pi_a \Gamma_b + \Lambda_{a, b} - \widetilde{\Lambda}_{a, b}, \label{E:Th_bPi_a}\\
		\Theta_b(\Pi^*_a) &= \Lambda_{b, a} - \Pi_b^* \Pi_a^* - \Gamma_b \Pi^*_a - \Pi_b \Pi^*_a, \label{E:Th_bPi*_a}
	\end{align}
where:
	\begin{align}
	\Lambda_{a,b}f &\defeq \sum_{Q\in\mcd; \ep,\eta\neq 1} \widehat{a}(Q,\ep) \sum_{R\in\mcd, R \supset Q} \widehat{b}(R,\eta) \widehat{f}(R,\eta) \frac{1}{|R|} h_Q^{\ep}, \label{E:Lb_abDef}\\
	\textup{and} \qquad 
		\widetilde{\Lambda}_{a,b} &\defeq \sum_{Q\in\mcd; \ep,\eta\neq 1} \widehat{a}(Q,\ep) \widehat{b}(Q,\eta) \widehat{f}(Q,\eta) \frac{1}{|Q|} h_Q^{\ep}.
	\end{align}
	We remark to the reader, that we are using a slightly different definition of $\Lambda_{a,b}$ than in \cite{HLW2}.  The $\Lambda_{a,b}$ defined in \eqref{E:Lb_abDef} corresponds to $\Lambda^*_{b,a}$ in \cite{HLW2}. However, we shall see later that it is more advantageous for our purposes to work with the definitions above.

We claim that:

\begin{lm} \label{L:Lambda_ab}
Let $\mu, \lb \in A_p$ for some $1<p<\infty$. Suppose $a \in BMO^2_{\mcd}$ and $b \in BMO^2_{\mcd}(\nu)$, where $\mcd$ is a fixed dyadic grid on $\mbr^n$ and $\nu = \mu^{\frac{1}{p}}\lb^{-\frac{1}{p}}$. If $T$ is any one of the operators:
	$$\Lambda_{a, b},\: \widetilde{\Lambda}_{a,b},\: \Lambda_{b, a},\: \widetilde{\Lambda}_{b, a},\: \Theta_b(\Pi_a) \text{, or } \Theta_b(\Pi^*_a),$$
then:
	\begin{equation} \label{E:Lb_2WtBd}
	\|T\|_{L^p(\mu;\: \lb)} \lesssim C(\mu, \lb, p) \|a\|_{BMO_{\mcd}^2} \|b\|_{BMO^2_{\mcd}(\nu)}.
	\end{equation}
In particular, if $\mu = \lb = w \in A_2$:
	\begin{equation} \label{E:Lb_1WtBd}
	\|T\|_{L^2(w)} \lesssim  \|a\|_{BMO^2_{\mcd}} \|b\|_{BMO^2_{\mcd}} [w]_{A_2}^2.
	\end{equation}
\end{lm}

Proposition \ref{P:Comm1} with $i = j = 0$ and $k=1$ follows from this immediately, by the assumptions on the $BMO$ norms of $a$ and $d$. Before we proceed with the proof, we remark that, while \eqref{E:Lb_2WtBd} was proved in some form for the $\Lambda$ operators in \cite{HLW2}, we need a slight modification of that proof in order for it to yield the one-weight result. Roughly speaking, the original proof boils down to  $\|S_{\mcd}\|_{L^p(\mu)} \|M \Pi_a^* \|_{L^q(\lb')}$, which in the one-weight case gives a factor of $[w]^3_{A_2}$.


\begin{proof}[Proof of Lemma \ref{L:Lambda_ab}]

It suffices to prove the result for the $\Lambda$ operators. For then, from the decomposition in \eqref{E:Th_bPi_a}:
\begin{align*}
\| \Theta_b(\Pi_a)\|_{L^p(\mu;\:\lb)} & \leq \|\Pi_a\|_{L^p(\lb)} \| \Pi_b + \Gamma_b \|_{L^p(\mu;\:\lb)} + \| \Lambda_{a, b} - \widetilde{\Lambda}_{a, b}\|_{L^p(\mu;\:\lb)}\\
	& \lesssim C(\mu, \lb, p) \|a\|_{BMO^2_{\mcd}} \|b\|_{BMO^2_{\mcd}(\nu)},
\end{align*}
where we used \eqref{E:Paraprod1WtBds} for the $\Pi_a$ term, and Theorem \ref{T:Paraprod2WtBds} for the paraproducts with symbol $b$. Similarly, we can see from \eqref{E:Th_bPi*_a} that $\Theta_b(\Pi^*_a)$ obeys the same bound. If $\mu=\lb=w\in A_2$:
$$ \|\Theta_b(\Pi_a)\|_{L^2(w)} \leq \|\Pi_a\|_{L^2(w)} \|\Pi_b + \Gamma_b\|_{L^2(w)} + \|\Lambda_{a, b} + \widetilde{\Lambda}_{a, b}\|_{L^2(w)}
	\lesssim \|a\|_{BMO^2_{\mcd}} \|b\|_{BMO^2_{\mcd}} [w]^2_{A_2},$$
and the same holds for $\Theta_b(\Pi^*_a)$.

Now let us look at $\Lambda_{a,b}$. Let $f \in L^p(\mu)$ and $g \in L^q(\lb')$. Then
	$$\La \Lambda_{a,b}f, g\Ra = \sum_{Q\in\mcd, \ep\neq 1} \widehat{a}(Q,\ep) \widehat{g}(Q,\ep) \sum_{R\in\mcd, R\supset Q; \eta\neq 1} \widehat{b}(R,\eta) \widehat{f}(R,\eta) \frac{1}{|R|}.$$
Let $a_{\tau} = T_{\tau}a$ and $b_{\sigma} = T_{\sigma}b$, where $T_{\tau}$ and $T_{\sigma}$ are martingale transforms with $\tau_{Q,\ep} = \pm 1$ and $\sigma_{Q, \ep} = \pm 1$ chosen for every pair $(Q\in\mcd, \ep\neq 1)$ such that 
	$$\tau_{Q,\ep}\widehat{a}(Q,\ep) \widehat{g}(Q,\ep) \geq 0 \text{, and } \sigma_{Q,\ep} \widehat{b}(Q,\ep)\widehat{f}(Q,\ep) \geq 0.$$	
Then
	\begin{align}
	\nonumber	\left| \La \Lambda_{a,b}f, g\Ra \right| & \leq \sum_{Q\in\mcd;\ep\neq 1} \widehat{a_{\tau}}(Q,\ep) \widehat{g}(Q,\ep) 
		\sum_{R\in\mcd, R\supset Q ; \eta\neq 1} \widehat{b_{\sigma}}(R,\eta) \widehat{f}(R,\eta) \frac{1}{|R|} \\
	\label{E:LbBd1}	&\leq \sum_{Q\in\mcd;\ep\neq 1} \widehat{a_{\tau}}(Q,\ep) \widehat{g}(Q,\ep) \La \Pi^*_{b_{\sigma}} f\Ra_{Q},
	\end{align}
where the last inequality follows from 	
	$$\La \Pi^*_{b_{\sigma}} f\Ra_Q = \sum_{P\in\mcd, P \subsetneq Q; \ep\neq 1} \widehat{b_{\sigma}}(P,\ep) \widehat{f}(P,\ep) \frac{1}{|Q|} +
		\sum_{R\in\mcd, R\supset Q; \eta\neq 1} \widehat{b_{\sigma}}(R,\eta) \widehat{f}(R,\eta) \frac{1}{|R|}.$$	
By the assumptions on $\sigma$ and $\tau$ and the Monotone Convergence Theorem, \eqref{E:LbBd1} becomes:
	\begin{align*}
	\sum_{Q\in\mcd,\ep\neq 1} \int \widehat{a_{\tau}}(Q,\ep) \widehat{g}(Q,\ep) (\Pi^*_{b_{\sigma}}f)(x) \frac{\unit_Q(x)}{|Q|}\,dx &= 
		\int \sum_{Q\in\mcd,\ep\neq 1} \widehat{a_{\tau}}(Q,\ep) \widehat{g}(Q,\ep) (\Pi^*_{b_{\sigma}}f)(x) \frac{\unit_Q(x)}{|Q|}\,dx\\
	&= \La \Pi^*_{a_{\tau}}g, \Pi^*_{b_{\sigma}}f \Ra. 
	\end{align*}
Therefore, by \eqref{E:Paraprod1WtBds} and Theorem \ref{T:Paraprod2WtBds},
	\begin{align}
	\label{E:LbL1} \| \Lambda_{a,b} \|_{L^p(\mu;\:\lb)}  & \leq  \|\Pi^*_{a_{\tau}}\|_{L^q(\lb')} \|\Pi^*_{b_{\sigma}}\|_{L^p(\mu;\: \lb)} \\
	\nonumber &\lesssim C(\mu, \lb, p) \|a_{\tau}\|_{BMO^2_{\mcd}} \|b_{\sigma}\|_{BMO^2_{\mcd}(\nu)}\\
	\nonumber & \lesssim C(\mu, \lb, p) \|a\|_{BMO^2_{\mcd}} \|b\|_{BMO^2_{\mcd}(\nu)},
	\end{align}
where the last inequality follows from Proposition \ref{P:BMOSigma}. Letting $\mu = \lb = w \in A_2$ in \eqref{E:LbL1}:
	$$ \| \Lambda_{a,b} \|_{L^2(w)} \lesssim \|a_{\tau}\|_{BMO^2_{\mcd}}[w]_{A_2} \|b_{\sigma}\|_{BMO^2_{\mcd}} [w]_{A_2} \lesssim \|a\|_{BMO^2_{\mcd}} \|b\|_{BMO^2_{\mcd}} [w]^2_{A_2},$$
which proves the result for $\Lambda_{a, b}$. An identical argument proves the result for $\widetilde{\Lambda}_{a,b}$.  As for $\Lambda_{b,a}$ or $\widetilde{\Lambda}_{b,a}$, the argument follows similarly, with a few modifications:
	$$\| \Lambda_{b, a}\|_{L^p(\mu;\:\lb)} \leq \|\Pi^*_{a_{\tau}}\|_{L^p(\mu)} \|\Pi^*_{b_{\sigma}}\|_{L^q(\lb';\: \mu')}.$$
\end{proof}


\section{The Second Iteration $\big[b, [b, \mbs^{ij}]\big]$} \label{S:SecondComm}

In this section we take a closer look at what happens when $k = 2$, and develop the rest of the tools we need for the general case. 
From \eqref{E:CommTh_bDec} and \eqref{E:Comm1_Dec}:
	$$C_b^2(\mbs^{ij}) = [b, C_b^1(\mbs^{ij})] = [\mfp_b, C_b^1(\mbs^{ij})] + \Theta_b\big( C_b^1(\mbs^{ij}) \big).$$
We show that each term obeys the bounds in Theorem \ref{T:BigThm_Sij}. The first term can be bounded using \eqref{E:Paraprod1WtBds} and Theorem \ref{T:BigThm_Sij} with $k = 1$.  In order to analyze the second term, we look at some simple properties of $\Theta_b$ that will be useful. Suppose $S$ and $T$ are some operators. Obviously $\Theta_b$ is linear, that is $\Theta_b(S + cT) = \Theta_b(S) + c\Theta_b(T)$. Moreover:
	\begin{equation} \label{E:Th_bComp}
	\Theta_b(ST) = \Theta_b(S)T + S\Theta_b(T).
	\end{equation}
To see this, note that we can write:
	$\Theta_b(ST)f = \Pi_{S(Tf)}b - S\Pi_{Tf}b + S\Pi_{Tf}b - ST\Pi_f b.$
In turn, this yields
	\begin{equation} \label{E:Th_bComm}
	\Theta_b \big( [S, T] \big) = [\Theta_b(S), T] + [S, \Theta_b(T)].
	\end{equation}
Then
	\begin{equation} \label{E:Comm2Temp1} 
	\Theta_b\big( C_b^1(\mbs^{ij})\big) = [\Theta_b(\mfp_b), \mbs^{ij}] + [\mfp_b, \Theta_b(\mbs^{ij})] + \Theta_b^2(\mbs^{ij}).
	\end{equation}
The second term in the expression is easily controlled using Proposition \ref{P:Comm1} with $k = 1$. For the first term, remark that
	\begin{equation} \label{E:Th_bGmb}
	\Theta_b(\Gamma_b) = 0,
	\end{equation}
which is easily seen by verifying that
	$$\Pi_{\Gamma_bf}b = \Gamma_b \Pi_f b = \sum_{R\in\mcd} \sum_{\ep,\eta\neq 1; \ep\neq \eta} \widehat{b}(R,\eta) \widehat{f}(R,\eta) \La b\Ra_{R} \frac{1}{\sqrt{|R|}} h_R^{\ep+\eta}.$$
So
	$\Theta_b(\mfp_b) = \Theta_b(\Pi_b) + \Theta_b(\Pi^*_b)$,
and both these terms can be bounded using Lemma \ref{L:Lambda_ab} with $a = b$. To analyze $\Theta_b^2(\mbs^{ij})$ and prove Proposition \ref{P:Comm1} for $k = 2$, we again need to look at the cancellative and non-cancellative cases separately.

\subsection{The Cancellative Case.} \label{Ss:Cb2_ij} Using the expression for $\Theta_b(\mbs^{ij})$ in \eqref{E:ThbSij}, we find that
	$$\Pi_{\Theta_b(\mbs^{ij})f} b = \sum_{R\in\mcd; \ep,\eta\neq 1} \sum_{P\in R_{(i)}, Q\in R_{(j)}} a^{\ep\eta}_{PQR} \widehat{f}(P, \ep) 
		\big( \La b\Ra_Q - \La b\Ra_P \big) \La b\Ra_Q h_Q^{\eta},$$
and
	$$\Theta_b(\mbs^{ij}) \Pi_f b = \sum_{R\in\mcd; \ep,\eta\neq 1} \sum_{P \in R_{(i)}, Q \in R_{(j)}} a^{\ep\eta}_{PQR} \widehat{f}(P, \ep)
		\big( \La b \Ra_Q - \La b \Ra_P \big) \La b \Ra_P h_Q^{\eta}.$$
Therefore
	$$\Theta_b^2 (\mbs^{ij})f = \sum_{R\in\mcd; \ep,\eta\neq 1} \sum_{P\in R_{(i)}, Q\in R_{(j)}} a^{\ep\eta}_{PQR} \widehat{f}(P,\ep) 
		\big( \La b\Ra_Q - \La b\Ra_P \big)^2 h_Q^{\eta}.$$
Now consider the operator 
	$$U_{(j)}f \defeq \sum_{Q\in\mcd; \eta\neq 1} \big( \La b\Ra_Q - \La b\Ra_{Q^{(j)}} \big) \widehat{f}(Q,\eta) h_Q^{\eta},$$
for a non-negative integer $j$. Then
	$$U_{(j)} \Theta_b(\mbs^{ij}) f = \sum_{R\in\mcd; \ep,\eta\neq 1} \sum_{P \in R_{(i)}, Q\in R_{(j)}} a^{\ep\eta}_{PQR} \widehat{f}(P, \ep) 
		\big( \La b\Ra_Q - \La b\Ra_P \big) \big( \La b\Ra_Q - \La b\Ra_R \big) h_Q^{\eta},$$
and
	$$\Theta_b(\mbs^{ij}) U_{(i)} f = \sum_{R\in\mcd; \ep,\eta\neq 1} \sum_{P \in R_{(i)}, Q\in R_{(j)}} a^{\ep\eta}_{PQR} \widehat{f}(P, \ep)
		\big( \La b\Ra_Q - \La b\Ra_P \big) \big( \La b\Ra_P - \La b\Ra_R \big) h_Q^{\eta}.$$ 
So 
	$$\Theta_b^2(\mbs^{ij}) = U_{(j)} \Theta_b(\mbs^{ij}) - \Theta_b(\mbs^{ij}) U_{(i)}.$$
We claim that for any $A_2$ weight $w$:
	\begin{equation} \label{E:UjBound}
	\left\| U_{(j)} \right\|_{L^2(w)} \lesssim j [w]_{A_2} \|b\|_{BMO^2_{\mcd}}.
	\end{equation}
To see this, remark that
	$$ \left| \La b\Ra_Q - \La b\Ra_{Q^{(j)}} \right| \leq j 2^n \|b\|_{BMO^2_{\mcd}},$$
so $U_{(j)}$ can be expressed as 
	$$U_{(j)} = j 2^n \|b\|_{BMO^2_{\mcd}} T_{\sigma},$$
where $T_{\sigma}$ is a martingale transform. Then \eqref{E:UjBound} follows from \eqref{E:MartTBd}. Finally, this and Proposition \ref{P:Comm1} with $k=1$ give that
	\begin{align}
	\left\| \Theta_b^2(\mbs^{ij}) \right\|_{L^p(\mu;\:\lb)} & \leq \|\Theta_b(\mbs^{ij}) \|_{L^p(\mu;\:\lb)} \big( \|U_{(i)}\|_{L^p(\mu)} + \|U_{(j)}\|_{L^p(\lb)} \big) \\
		&\lesssim \kappa^2_{ij} C(\mu, \lb, p) \|b\|_{BMO^2_{\mcd}(\nu)} \|b\|_{BMO^2_{\mcd}},
	\end{align}
and, in the one-weight case,
	$$\left\| \Theta_b^2(\mbs^{ij}) \right\|_{L^2(w)} \lesssim \kappa_{ij}^2 \|b\|^2_{BMO^2_{\mcd}} [w]^3_{A_2}.$$

\subsection{The case $i = j = 0$.} \label{Ss:Cb2_00}

From \eqref{E:00HRT}:
	$$\Theta_b^2(\mbs^{00}) = \Theta_b^2(\Pi_a) + \Theta_b^2(\Pi^*_{d}).$$
Using the expression in \eqref{E:Th_bPi_a} and the properties of $\Theta_b$ in \eqref{E:Th_bComp} and \eqref{E:Th_bGmb}:
	\begin{equation} \label{E:Cb00Temp1}
	\Theta_b^2(\Pi_a) = \Theta_b(\Pi_a) \Pi_b + \Pi_a \Theta_b(\Pi_b) + \Theta_b(\Pi_a) \Gamma_b + \Theta_b(\Lambda_{a,b}) - \Theta_b(\widetilde{\Lambda}_{a, b}).
	\end{equation}
Lemma \ref{L:Lambda_ab} and the paraproduct norms immediately control the first three terms, showing that their norms as operators $L^p(\mu) \to L^p(\lb)$ are bounded (up to a constant) by 
$\|a\|_{BMO^2_{\mcd}} \|b\|_{BMO^2_{\mcd}} \|b\|_{BMO^2_{\mcd}(\nu)}$, and their norms as operators $L^2(w) \to L^2(w)$ are bounded (up to a constant) by 
$\|a\|_{BMO^2_{\mcd}} \|b\|^2_{BMO^2_{\mcd}} [w]^3_{A_2}$. For the last two terms, we look at an interesting property of the $\Lambda_{a,b}$ operators:


\begin{prop} \label{P:Switch}
For some locally integrable functions $a$, $b$, $c$, the operator $\Lambda_{a,b}$ satisfies:
	\begin{align}
	\label{E:Switch} \Theta_c(\Lambda_{a,b}) &= \Pi_a \Lambda_{c,b}    \\
	\text{and }  \Theta_c(\widetilde{\Lambda}_{a,b}) &= 0.
	\end{align}
\end{prop}

\begin{proof}
To prove the first statement, note that
	\begin{align}
	\nonumber \La \Lambda_{a,b}f\Ra_Q &= 	\sum_{\substack{Q,R\in\mcd; R \supsetneq Q \\ \eta\neq 1}} \widehat{b}(R,\eta) \widehat{f}(R,\eta) \frac{1}{|R|} \sum_{\substack{P\in\mcd; Q \subsetneq P\subset R \\ \ep\neq 1}} \widehat{a}(P,\ep) h_P^{\ep}(Q)\\
	 \label{E:SwitchTemp1}
	 &= \sum_{\substack{Q,R\in\mcd; R \supsetneq Q \\ \eta\neq 1}} \widehat{b}(R,\eta) \widehat{f}(R,\eta) \frac{1}{|R|} \left( \La a\Ra_Q - \La a\Ra_R \right).
	\end{align}
A quick calculation shows that
	$$\Theta_c(\Lambda_{a,b}) f = \sum_{\substack{Q\in\mcd \\ \ep\neq 1}} \widehat{a}(Q,\ep)
		\left[ \sum_{\substack{Q,R\in\mcd; R \supset Q\\ \eta\neq 1}} \widehat{b}(R,\eta) \widehat{f}(R,\eta) \frac{1}{|R|} \big( \La c\Ra_Q - \La c\Ra_R \big) \right] h_Q^{\ep}.$$
From \eqref{E:SwitchTemp1}, we recognize the term in parentheses as $\La \Lambda_{c,b} f \Ra_Q$, and so
	$$\Theta_c(\Lambda_{a,b}) f = \sum_{Q\in\mcd, \ep\neq 1} \widehat{a}(Q,\ep) \La \Lambda_{c,b} f \Ra_Q h_Q^{\ep}  = \Pi_a \Lambda_{c,b}f.$$
The second statement follows by
	$$\Pi_{\widetilde{\Lambda}_{a,b}f}c = \widetilde{\Lambda}_{a,b}\Pi_f c = \sum_{Q\in\mcd; \ep,\eta\neq 1} \widehat{a}(Q,\ep) \widehat{b}(Q,\eta) \widehat{f}(Q,\eta) \frac{1}{|Q|} \La c\Ra_Q h_Q^{\ep}.$$
\end{proof}


Returning to \eqref{E:Cb00Temp1}, we can now see that the last two terms in the expression become simply $\Pi_a \Lambda_{b,b}$, which is controlled exactly as the other terms. The result for $\Theta_b^2(\Pi^*_a)$ follows similarly, after noting that $\Theta_b(\Lambda_{b,a}) = \Pi_b \Lambda_{b, a}$. Finally, recall the assumptions on the $BMO$ norms of $a$ and $d$ in \eqref{E:00HRT} and see that the results in this section prove Proposition \ref{P:Comm1} for $k = 2$ and $i = j = 0$.


\section{The General Case of Higher Iterations} \label{S:General}

In this section we prove Theorem \ref{T:BigThm_Sij}. A closer look at recursively expanding the formula:
	$$C_b^{m+1}(T) = [\mfp_b, C_b^m(T)] + \Theta_b\big( C_b^m(T) \big),$$
for some operator $T$, shows that, in order to control $C_b^{m+1}(T)$, we need not only bound the previous iterations $C_b^{m+1-k}(T)$, but really
	$$C_b^{m-k}(T),\: \Theta_b\big( C_b^{m-k}(T) \big),\: \Theta^2_b\big( C_b^{m-k}(T) \big), \ldots, \Theta^k_b\big( C_b^{m-k}(T) \big),$$
for every $0 \leq k \leq m-1$. So, it makes sense to instead prove the following more general statement:

\begin{thm} \label{T:BIGTHM_Shifts}
Under the same assumptions as Theorem \ref{T:BigThm_Sij}, for any integer $k \geq 1$:
\begin{equation} \label{ThbM_2Wt}
\left\| \Theta_b^M\big( C_b^k(\mbs^{ij}) \big) \right\|_{L^p(\mu;\:\lb)} \leq c \kappa_{ij}^{M+k} \|b\|_{BMO^2_{\mcd}}^{M+k-1} \|b\|_{BMO^2_{\mcd}(\nu)} \text{, for all } M \geq 0,
\end{equation}
where $c$ is a constant depending on $n$, $\mu$, $\lb$, $p$, $M$, and $k$. In particular, if $\mu = \lb = w \in A_2$,
\begin{equation} \label{ThbM_1Wt}
\left\| \Theta_b^M\big( C_b^k(\mbs^{ij}) \big) \right\|_{L^2(w)} \leq c \kappa_{ij}^{M+k} \|b\|^{M+k}_{BMO^2_{\mcd}} [w]_{A_2}^{M+k+1}  \text{, for all } M \geq 0,
\end{equation}
where $c$ is a constant depending on $n$, $M$, and $k$.
\end{thm}

Theorem \ref{T:BigThm_Sij} will then follow as a special case of the above result, with $M = 0$. We begin by completing the proof of Proposition \ref{P:Comm1}.


\begin{proof}[Proof of Proposition \ref{P:Comm1}]
So far, this result has been proved for $k = 1$ and $k = 2$. In case $(i,j) \neq (0,0)$, we generalize the argument in Section \ref{Ss:Cb2_ij}. We claim that for all $k \geq 2$ and $(i,j) \neq (0,0)$:
	\begin{align}
	\label{E:Temp1} \Theta_b^k(\mbs^{ij}) f &= \sum_{R\in\mcd; \ep,\eta\neq 1} \sum_{P\in R_{(i)}, Q\in R_{(j)}} a^{\ep\eta}_{PQR} \widehat{f}(P,\ep) 
		\big(\La b\Ra_Q - \La b\Ra_P\big)^k h_Q^{\eta}\\
	\label{E:Temp2} &= U_{(j)} \Theta_b^{k-1}(\mbs^{ij}) - \Theta_b^{k-1} (\mbs^{ij})U_{(i)}.
	\end{align}	
Then, assuming Proposition \ref{P:Comm1} holds for some $k\geq 1$, the result for $k+1$ follows from \eqref{E:UjBound}.
To see this, assume \eqref{E:Temp1} holds for some $k \geq 2$. Then
	$$\Pi_{\Theta^k_b(\mbs^{ij})f} b = \sum_{R\in\mcd; \ep,\eta\neq 1} \sum_{P\in R_{(i)}, Q\in R_{(j)}} a^{\ep\eta}_{PQR} \widehat{f}(P, \ep) 
		\big( \La b\Ra_Q - \La b\Ra_P \big)^k \La b\Ra_Q h_Q^{\eta},$$
and
	$$\Theta^k_b(\mbs^{ij}) \Pi_f b = \sum_{R\in\mcd; \ep,\eta\neq 1} \sum_{P \in R_{(i)}, Q \in R_{(j)}} a^{\ep\eta}_{PQR} \widehat{f}(P, \ep)
		\big( \La b \Ra_Q - \La b \Ra_P \big)^k \La b \Ra_P h_Q^{\eta}.$$
Since $\Theta_b^{k+1}(\mbs^{ij})f = \Pi_{\Theta^k_b(\mbs^{ij})f} b - \Theta^k_b(\mbs^{ij}) \Pi_f b$, we see that \eqref{E:Temp1} holds for $k+1$. Similarly,
	$$U_{(j)} \Theta^k_b(\mbs^{ij}) f = \sum_{R\in\mcd; \ep,\eta\neq 1} \sum_{P \in R_{(i)}, Q\in R_{(j)}} a^{\ep\eta}_{PQR} \widehat{f}(P, \ep) 
		\big( \La b\Ra_Q - \La b\Ra_P \big)^k \big( \La b\Ra_Q - \La b\Ra_R \big) h_Q^{\eta},$$
and
	$$\Theta^k_b(\mbs^{ij}) U_{(i)} f = \sum_{R\in\mcd; \ep,\eta\neq 1} \sum_{P \in R_{(i)}, Q\in R_{(j)}} a^{\ep\eta}_{PQR} \widehat{f}(P, \ep)
		\big( \La b\Ra_Q - \La b\Ra_P \big)^k \big( \La b\Ra_P - \La b\Ra_R \big) h_Q^{\eta},$$
from which  \eqref{E:Temp2} with $k+1$ follows.

For the case $i = j = 0$,
	$$\Theta_b^k(\mbs^{00}) = \Theta_b^k(\Pi_a) + \Theta_b^k(\Pi^*_d),$$
with $\|a\|_{BMO^2_{\mcd}} \lesssim 1$ and $\|d\|_{BMO^2_{\mcd}} \lesssim 1$. Proposition \ref{P:Comm1} for $i = j =0$ therefore follows trivially from the next result.
\end{proof}



\begin{prop} \label{P:Th_b^k(P,Lb)}
Under the same assumptions as Theorem \ref{T:BigThm_Sij}, let $P_a$ denote either one of the operators $\Pi_a$ and $\Pi^*_a$, and $\Lambda$ denote either one of the operators $\Lambda_{a,b}$ or $\Lambda_{b, a}$. Then for any integer $k \geq 1$:
	\begin{equation} \label{E:Thbk(P)2Wt}
		\left\| \Theta_b^k(P_a) \right\|_{L^p(\mu;\:\lb)} \leq c \|a\|_{BMO^2_{\mcd}} \|b\|^{k-1}_{BMO^2_{\mcd}} \|b\|_{BMO^2_{\mcd}(\nu)},
	\end{equation}
	\begin{equation} \label{E:Thbk(L)2Wt}
		\left\| \Theta_b^k (\Lambda) \right\|_{L^p(\mu;\:\lb)} \leq c \|a\|_{BMO^2_{\mcd}} \|b\|^k_{BMO^2_{\mcd}} \|b\|_{BMO^2_{\mcd}(\nu)},
	\end{equation}
where $c$ is a constant depending on $n$, $p$, $\mu$, $\lb$, and $k$. In particular, if $\mu = \lb = w \in A_2$:
	\begin{equation} \label{E:Thbk(P)1Wt}
		\left\| \Theta_b^k(P_a) \right\|_{L^2(w)} \leq c \|a\|_{BMO^2_{\mcd}} \|b\|^k_{BMO^2_{\mcd}} [w]^{k+1}_{A_2},
	\end{equation}
	\begin{equation} \label{E:Thbk(L)1Wt}
		\left\| \Theta_b^k (\Lambda) \right\|_{L^2(w)} \leq c \|a\|_{BMO^2_{\mcd}} \|b\|^{k+1}_{BMO^2_{\mcd}} [w]^{k+2}_{A_2},
	\end{equation}
where $c$ is a constant depending on $n$ and $k$.
\end{prop}

\begin{proof}
This result with $k = 1$ was proved in Lemma \ref{L:Lambda_ab} for $\Theta_b(P_a)$, and in Section \ref{Ss:Cb2_00} for $\Theta_b(\Lambda)$. We proceed by (strong) induction.  Fix $m \geq 1$  and suppose Proposition \ref{P:Th_b^k(P,Lb)} holds for all $1 \leq k \leq m$. We show that it then holds for $k = m + 1$.

Let us look at the case $P_a = \Pi_a$:
	$$\Theta_b^{m+1}(\Pi_a) = \Theta_b^m \big( \Theta_b(\Pi_a) \big) = \Theta_b^m \big(\Pi_a (\Pi_b + \Gamma_b) \big) + \Theta_b^m(\Lambda_{a,b}).$$ 
Remark that the last term is already controlled by the induction assumption. To analyze the first term, we use the binomial formula:
	\begin{equation} \label{E:BinomComp}
	\Theta_b^m(ST) = \sum_{k = 0}^m \binom{m}{k} \Theta_b^{m-k}(S) \Theta_b^k(T),
	\end{equation}
which follows from \eqref{E:Th_bComp} by a simple induction argument. Then
	\begin{equation} \label{E:Ind1t1} 
	\left\| \Theta_b^m\big( \Pi_a(\Pi_b + \Gamma_b) \big) \right\|_{L^p(\mu;\:\lb)} \leq \sum_{k = 0}^m \binom{m}{k}   
		\left\| \Theta_b^{m-k}(\Pi_a) \right\|_{L^p(\lb)} \left\| \Theta_b^k(\Pi_b + \Gamma_b) \right\|_{L^p(\mu;\:\lb)}.
	\end{equation}
The statement:
	\begin{equation} \label{E:Ind1t3}
	\left\| \Theta_b^{m-k} (\Pi_a) \right\|_{L^2(w)} \lesssim C(m-k) \|a\|_{BMO^2_{\mcd}} \|b\|^{m-k}_{BMO^2_{\mcd}} [w]_{A_2}^{m-k+1} \text{, } 0 \leq k \leq m, 
	\end{equation}
for all $w \in A_2$, follows from the induction assumption on $P_a$ for $0 \leq k \leq m-1$, and from \eqref{E:Paraprod1WtBds} for $k=m$. Then, by Extrapolation,	
	\begin{equation} \label{E:Ind1t4} 
	\left\| \Theta_b^{m-k}(\Pi_a) \right\|_{L^p(\lb)} \lesssim C(\lb, p, m-k) \|a\|_{BMO^2_{\mcd}} \|b\|_{BMO^2_{\mcd}}^{m-k} \text{, } 0\leq k \leq m.
	\end{equation}
On the other hand, noting that $\Theta_b^0(\Pi_b + \Gamma_b) = \Pi_b + \Gamma_b$ and $\Theta_b^k(\Pi_b + \Gamma_b) = \Theta_b^k(\Pi_b)$ for $1 \leq k \leq m$, we have:
	$$ \left\|\Theta_b^k(\Pi_b + \Gamma_b) \right\|_{L^p(\mu;\:\lb)} \lesssim C(\mu, \lb, p, k) \|b\|^k_{BMO^2_{\mcd}} \|b\|_{BMO^2_{\mcd}(\nu)} \text{, } 0 \leq k \leq m,$$	
which follows from Theorem \ref{T:Paraprod2WtBds} for $k = 0$, and from the induction assumption on $P_a$ with $a = b$ for $1 \leq k \leq m$. Similarly:
	$$ \left\| \Theta_b^k(\Pi_b + \Gamma_b) \right\|_{L^2(w)} \lesssim C(k) \|b\|_{BMO^2_{\mcd}}^{k+1} [w]_{A_2}^{k+1} \text{, } 0\leq k \leq m.$$	
From these estimates and \eqref{E:Ind1t1}, we obtain
	$$ \left\| \Theta_b^m\big( \Pi_a(\Pi_b + \Gamma_b) \big) \right\|_{L^p(\mu;\:\lb)} \lesssim C(\mu, \lb, p, m) \|a\|_{BMO^2_{\mcd}} \|b\|^m_{BMO^2_{\mcd}} \|b\|_{BMO^2_{\mcd}(\nu)},$$
and
	$$ \left\| \Theta_b^m\big( \Pi_a(\Pi_b + \Gamma_b) \big) \right\|_{L^2(w)} \lesssim C(m) \|a\|_{BMO^2_{\mcd}} \|b\|^{m+1}_{BMO^2_{\mcd}} [w]_{A_2}^{m+2},$$
which proves the result for $k = m+1$ and $P_a = \Pi_a$. The case $P_a = \Pi^*_a$ follows similarly.

Now suppose $\Lambda = \Lambda_{a,b}$. Then, by \eqref{E:Switch},
	$$ \Theta_b^{m+1}(\Lambda_{a,b}) = \Theta_b^m\big( \Theta_b(\Lambda_{a, b}) \big) = \Theta_b^m(\Pi_a \Lambda_{b,b}).$$
Using the binomial formula again,
	\begin{equation} \label{E:Ind1t2} 
	\left\| \Theta_b^{m+1}(\Lambda_{a,b}) \right\|_{L^p(\mu;\:\lb)} \leq \sum_{k=0}^m \binom{m}{k}
		\left\| \Theta_b^{m-k}(\Pi_a) \right\|_{L^p(\lb)}  \left\| \Theta_b^k(\Lambda_{b,b}) \right\|_{L^p(\mu;\:\lb)}. 
	\end{equation}
The statements:
	$$ \left\| \Theta_b^k(\Lambda_{b,b}) \right\|_{L^p(\mu;\:\lb)} \lesssim C(\mu,\lb,p,k) \|b\|^{k+1}_{BMO^2_{\mcd}} \|b\|_{BMO^2_{\mcd}(\nu)} \text{, } 0 \leq k \leq m, $$
	$$ \left\| \Theta_b^k(\Lambda_{b,b}) \right\|_{L^2(w)} \lesssim C(k) \|b\|^{k+2}_{BMO^2_{\mcd}}[w]^{k+2}_{BMO^2_{\mcd}} \text{, } 0 \leq k \leq m, $$
follow from the induction assumption on $\Lambda$ with $a = b$ for $1 \leq k \leq m$, and from Lemma \ref{L:Lambda_ab} for $k = 0$. Combining these with \eqref{E:Ind1t3} and \eqref{E:Ind1t4}, we have from \eqref{E:Ind1t2}:
	$$ \left\| \Theta_b^{m+1}(\Lambda_{a,b}) \right\|_{L^p(\mu;\:\lb)} \lesssim C(\mu,\lb,p,m+1) \|a\|_{BMO^2_{\mcd}} \|b\|^{m+1}_{BMO^2_{\mcd}} \|b\|_{BMO^2_{\mcd}(\nu)}, $$
	$$ \left\| \Theta_b^{m+1}(\Lambda_{a,b}) \right\|_{L^2(w)} \lesssim C(m+1) \|a\|_{BMO^2_{\mcd}} \|b\|^{m+2}_{BMO^2_{\mcd}} [w]^{m+3}_{A_2}, $$
which proves the result for $k = m+1$ and $\Lambda = \Lambda_{a,b}$. Since the case $\Lambda = \Lambda_{b,a}$ follows similarly, Proposition \ref{P:Th_b^k(P,Lb)} is proved.

\end{proof}

We now have all the tools needed to prove Theorem \ref{T:BIGTHM_Shifts}.

\begin{proof}[Proof of Theorem \ref{T:BIGTHM_Shifts}]
We prove the result for $k = 1$, so we look at
	$$\Theta_b^M\big( C_b^1(\mbs^{ij}) \big) = \Theta_b^M\big( [\mfp_b, \mbs^{ij}] + \Theta_b(\mbs^{ij}) \big),$$
for some integer $M \geq 0$. At this point, we use another easily deduced binomial formula:
	\begin{equation} \label{E:BinomComm}
	\Theta_b^M\big( [S, T] \big) = \sum_{m=0}^M \binom{M}{m} \big[ \Theta_b^{M-m}(S), \Theta_b^m(T) \big].
	\end{equation}
Then
	\begin{align*}
	\left\| \Theta_b^M\big( C_b^1(\mbs^{ij}) \big) \right\|_{L^p(\mu;\:\lb)} \leq & \sum_{m=0}^M 
		\binom{M}{m} \left\| \Theta_b^{M-m}(\mfp_b) \right\|_{L^p(\mu;\:\lb)} 
			\big( \|\Theta_b^m(\mbs^{ij})\|_{L^p(\mu)} +  \|\Theta_b^m(\mbs^{ij})\|_{L^p(\lb)} \big) \\
		& +	\left\| \Theta_b^{M+1}(\mbs^{ij}) \right\|_{L^p(\mu;\:\lb)}.
	\end{align*}
Applying Proposition \ref{P:Th_b^k(P,Lb)} with $a = b$ and Proposition \ref{P:Comm1}, we obtain the result for $k=1$.

Finally, suppose Theorem \ref{T:BIGTHM_Shifts} holds for some $k\geq 1$ and let an integer $M \geq 0$. Then
	$$\Theta_b^M\big( C_b^{k+1}(\mbs^{ij}) \big) = \Theta_b^M \Big( [\mfp_b, C_b^k(\mbs^{ij})] + \Theta_b\big( C_b^k(\mbs^{ij}) \big)\Big),$$
and the result for $k+1$ again follows from Propositions \ref{P:Th_b^k(P,Lb)}  and \ref{P:Comm1}.
\end{proof}



\begin{bibdiv}
\begin{biblist}

\normalsize



\bib{Beznosova}{article}{
	author={Beznosova, O.},
	title={Linear bound for the dyadic paraproduct on weighted Lebesgue space $L_2(w)$},
	journal={Journal of Functional Analysis},
	volume={255},
	number={4},
	date={2008},
	pages={994--1007},
}

\bib{Bloom}{article}{
   author={Bloom, Steven},
   title={A commutator theorem and weighted BMO},
   journal={Trans. Amer. Math. Soc.},
   volume={292},
   date={1985},
   number={1},
   pages={103--122}
}

\bib{Buckley}{article}{
	author={Buckley, S. M.},
	title={Estimates for operator norms and reverse Jensen’s inequalities},
	journal={Trans. Amer. Math. Soc.},
	volume={340},
	date={1993},
	number={1},
	pages={253--272},
}

\bib{Chung}{article}{
	author={Chung, Daewon},
	title={Weighted inequalities for multivariable dyadic paraproducts},
	journal={Publ. Mat.},
	volume={55},
	number={2},
	date={2011},
	pages={475--499},
}

\bib{ChungPereyraPerez}{article}{
	author={Chung, Daewon},
	author={Pereyra, M. Cristina},
	author={Perez, Carlos},
	title={Sharp Bounds for General Commutators On Weighted Lebesgue Spaces},
	journal={Transactions of the American Mathematical Society},
	volume={364},
	number={3},
	date={2012},
	pages={1163--1177},
}

\bib{CRW}{article}{
  author={Coifman, R. R.},
  author={Rochberg, R.},
  author={Weiss, Guido},
  title={Factorization theorems for Hardy spaces in several variables},
  journal={Ann. of Math. (2)},
  volume={103},
  date={1976},
  number={3},
  pages={611--635},
}

\bib{CruzClassicOps}{article}{
	author={Cruz-Uribe, D.},
	author={Moen, K.},
	title={Sharp norm inequalities for commutators of classical operators},
	journal={Publ. Mat.},
	volume={56},
	number={1},
	date={2011},
	pages={147--190},
}


\bib{DalencOu}{article}{
  author={Dalenc, L.},
  author={Ou, Y.},
  title={Upper Bound for Multi-parameter Iterated Commutators},
  pages={1--25},
  eprint={http://arxiv.org/abs/1401.5994},
  year={2014},
}

\bib{Extrapolation}{article}{
	author={Dragi\u{c}evi\'c, O.},
	author={Grafakos, L.},
	author={Pereyra, M. C.},
	author={Petermichl, S.},
	title={Extrapolation and sharp norm estimates for classical operators on weighted Lebegue spaces},
	journal={Publ. Math},
	volume={49},
	date={2005},
	number={1},
	pages={73--91},
}


\bib{HLW1}{article}{
    author={Holmes, Irina},
    author={Lacey, Michael T.},
    author={Wick, Brett D.},
    title={Bloom's Inequality: Commutators in a Two-Weight Setting},
    date={2015},
    eprint={http://arxiv.org/abs/1505.07947}
}


\bib{HLW2}{article}{
	author={Holmes, Irina},
	author={Lacey, Michael T.},
	author={Wick, Brett D.},
  	title={Commutators in the Two-Weight Setting},
  	date={2015},
	eprint={http://arxiv.org/abs/1506.05747},
}

\bib{HytRepOrig}{article}{
  author={Hyt{\"o}nen, T.},
  title={Representation of singular integrals by dyadic operators, and the A\_2 theorem},
  eprint={http://arxiv.org/abs/1108.5119},
  year={2011},
}

\bib{HytRep}{article}{
  author={Hyt{\"o}nen, T.},
  title={The sharp weighted bound for general Calder\'on-Zygmund operators},
  journal={Ann. of Math. (2)},
  volume={175},
  date={2012},
  number={3},
  pages={1473--1506},
}


\bib{HytLacey}{article}{
  author={Hyt{\"o}nen, T. P.},
  author={Lacey, M. T.},
  author={Martikainen, H.},
  author={Orponen, T.},
  author={Reguera, M},
  author={Sawyer, E. T.},
  author={Uriarte-Tuero, I.},
  title={Weak and strong type estimates for maximal truncations of Calder\'on-Zygmund operators on $A_p$ weighted spaces},
  journal={J. Anal. Math.},
  volume={118},
  date={2012},
  number={1},
  pages={177--220},
}

\bib{HytPerezTV}{article}{
  author={Hyt{\"o}nen, T.},
  author={P{\'e}rez, C.},
  author={Treil, S.},
  author={Volberg, A.},
  title={Sharp weighted estimates for dyadic shifts and the A2 conjecture},
  journal={Journal f{\"u}r die reine und angewandte Mathematik},
  volume={2014},
  number={687},
  pages={43--86},
  date={2012},
}


\bib{Lacey}{article}{
  author={Lacey, M. T.},
  title={On the $A_2$ inequality for Calder\'on-Zygmund operators},
  conference={ title={Recent advances in harmonic analysis and applications}, },
  book={ series={Springer Proc. Math. Stat.}, volume={25}, publisher={Springer, New York}, },
  date={2013},
  pages={235--242},
}

\bib{Muck}{article}{
  author={Muckenhoupt, Benjamin},
  title={Weighted norm inequalities for the Hardy maximal function},
  journal={Trans. Amer. Math. Soc.},
  volume={165},
  date={1972},
  pages={207--226}
}

\bib{MuckWheeden}{article}{
  author={Muckenhoupt, B.},
  author={Wheeden, R. L.},
  title={Weighted bounded mean oscillation and the Hilbert transform},
  journal={Studia Math.},
  volume={54},
  date={1975/76},
  number={3},
  pages={221--237},
}



\bib{TreilSharpA2}{article}{
  author={Treil, S.},
  title={Sharp $A_2$ estimates of Haar shifts via Bellman function},
  date={2011},
  pages={1-23},
  eprint={http://arxiv.org/abs/1105.2252},
}

\bib{Wittwer}{article}{
	author={Wittwer, Janine},
	title={A Sharp Estimate on the Norm of the Martingale Transform},
	journal={Mathematical Research Letters},
	volume={7},
	date={2000},
	pages={1--12},
}


\end{biblist}
\end{bibdiv}
\end{document}